\documentclass[10pt]{article}
\usepackage{amsfonts}
\usepackage{pdfsync}
\usepackage{amsfonts}  
\usepackage{pdfsync}
\title{Radial solutions of truncated Laplacian  equations  in punctured balls }
\author{Isabeau Birindelli \thanks{isabeau@mat.uniroma1.it}\\
Dipartimento di Matematica, Sapienza Universit\`a\  di Roma
\and
  Fran\c{c}oise Demengel\thanks{francoise.demengel@cyu.fr}\\
  D\'epartement de Math\'ematiques,
CY-Cergy Paris Universit\'e\  
  \and
   Fabiana  Leoni\thanks{leoni@mat.uniroma1.it}\\ 
   Dipartimento di Matematica, Sapienza Universit\`a\  di Roma}
   
\date{}

\usepackage{amsthm}
\usepackage{amsmath}
\usepackage{color}

\usepackage[colorlinks=true]{hyperref}

\catcode`@=11 \@addtoreset{equation}{section} \catcode`@=12

\newtheorem{theo}{Theorem}[section]
\newtheorem{prop}[theo]{Proposition}
\newtheorem{rema}[theo]{Remark}

\newtheorem{cor}[theo]{Corollary}
\newtheorem{lemme}[theo]{Lemma}

\def\Pkp{{ \cal P}_k^+} 
\def\Pkm{{ \cal P}_k^-}
\def\Pkpm{{ \cal P}_k^\pm}  
\def\R{\mathbb  R}

\def\ulambda{\underline{\lambda}_\gamma}

\begin{document}
\maketitle

\begin{abstract}
We consider equations involving the truncated laplacians $\Pkpm$ and having   lower order terms with singular potentials posed in punctured balls. We study both the principal eigenvalue problem  and the problem of classification of solutions, in dependence of their asymptotic behavior near the origin, for equations having also superlinear absorbing lower order terms. In case of  $\Pkp$, owing to  the mild degeneracy of the operator, we obtain results which are  analogous to the results for the Laplacian in dimension $k$. On the other hand, for operator $\Pkm$ we show that the strong degeneracy in ellipticity of the operator produces radically different results.
    \end{abstract} 
    
    \section{Introduction}

 This article is concerned with  equations related to the truncated Laplacian in punctured balls and having lower order terms with singular potentials. More precisely, we will consider both the existence of   radial  eigenfunctions for the eigenvalue problem 
\begin{equation}\label{eigen}
\left\{ 
\begin{array} {cl}
\displaystyle {\cal P}_k^\pm  ( D^2 u) + \lambda_\gamma \frac{u}{ r^\gamma} = 0 & \quad \hbox{ in } B(0,1)\setminus \{0\}\\[2ex]
u=0 & \quad \hbox{ on } \partial B(0,1)
\end{array}
\right.\end{equation}
where $\gamma >0$, and the existence and asymptotic behavior near the origin of  radial positive solutions of 
\begin{equation}\label{super}
 {\cal P}_k^\pm  ( D^2 u) + \mu \frac{u}{r^2} = u^p \qquad \hbox{ in } B(0,1)\setminus \{0\}\, ,
 \end{equation}
with $p>1$ and $0<\mu <\lambda_2$. $B(0,1)$ will always denote by the unit open ball centered at zero.

  We will observe that in the case of ${ \cal P}_k^+$ the obtained results are analogous to the results for the Laplacian in dimension $k$, and,  in particular,  for $\gamma >2$, the principal eigenvalue $\lambda_\gamma$ is $0$, while the obtained results for  ${ \cal P}_k^-$  deeply differ from the previous cases, and, in particular, the principal eigenvalue 
  $\lambda_\gamma$ is $+\infty $ for any $\gamma\geq0$. 
\medskip

   The operators $\Pkpm$, often referred to as  truncated Laplacians, are defined in the following way: for a symmetric matrix $X$, whose ordered eigenvalues are $\lambda_1\leq \lambda_2\cdots\leq \lambda_N$, and for an integer value $1\leq k\leq N$, one has
 $$\Pkm(X) :=\sum_{i=1}^k \lambda_i\ \ \mbox{and }\ \Pkp(X) :=\sum_{i=N-k+1}^N \lambda_i\, .$$
 When evaluated on the hessian matrix of the unknown function, the truncated laplacians are second order, degenerate elliptic operators which reduce to the standard laplacian for $k=N$. Because of their numerous applications, they have attracted some attention both in the PDE and the geometry world.  In particular,  in geometry they appear naturally when  considering manifolds of partially positive curvature, see \cite{Sha, Wu}, or  in problem concerning mean curvature flow in arbitrary codimension, see \cite{AS}. In their treatment of PDE from a  convex analysis point of view, Harvey and Lawson have studied the truncated Laplacian in different articles, see \cite{HL1,HL2}. For other results in PDE we refer e.g. to \cite{BG19, BGI18,  BGL,  CLN, G, GV, OS}.
Closer to the present work is \cite{BGI21}, where Birindelli, Galise and Ishii  introduced for $\Pkpm$ the notion of generalized eigenvalue  \`a la Berestycki, Nirenberg, Varadhan, see \cite{BNV}. In particular proving the existence of a positive eigenfunction when $k=1$ for strictly convex domains. For other $k>1$ the question of whether there exists or not an eigenfunction  is still open, even  for general strictly convex domains.  We will see that even here the case $k=1$ and the case $k>1$ differ in nature.
\medskip

 In the present paper we consider ${ \cal C}^2$ \emph{radial} solutions of problems \eqref{eigen} and \eqref{super}. As it is well known,  the eigenvalues of the hessian matrix of a smooth radial function $u(r)$ for $r>0$ are given by $\displaystyle \frac{u^\prime(r)}{r}$, with multiplicity $N-1$, and $u^{\prime\prime}(r)$. Therefore, for $1\leq k<N$, we get
 $$
 \Pkp(D^2u)= \left\{
 \begin{array}{ll}
 \displaystyle u^{\prime\prime} +(k-1)\frac{u^\prime}{r} &  \displaystyle \hbox{ if } u^{\prime\prime}\geq \frac{u^\prime}{r}\\[2ex]
\displaystyle k\frac{u^\prime}{r}&  \displaystyle \hbox{ if } u^{\prime\prime}\leq \frac{u^\prime}{r}
\end{array}\right.
$$
 and, analogously,
 $$
 \Pkm(D^2u)= \left\{
 \begin{array}{ll}
 \displaystyle u^{\prime\prime} +(k-1)\frac{u^\prime}{r} &  \displaystyle \hbox{ if } u^{\prime\prime}\leq \frac{u^\prime}{r}\\[2ex]
\displaystyle k\frac{u^\prime}{r}&  \displaystyle \hbox{ if } u^{\prime\prime}\geq \frac{u^\prime}{r}
\end{array}\right.
$$
So, when applied to smooth radial functions, the operators $\Pkpm$ switch from something like the laplacian in dimension $k$ to a first order operator, depending on the sign of $u^{\prime\prime}-\frac{u^\prime}{r}$. In order to focus on the strictly degenerate cases,  we will always suppose that $k<N$.
 
 When considering problems \eqref{eigen}, we are interested in proving the existence of positive solutions. In this case, the constants $\lambda_\gamma$ are referred to as the \emph{principal eigenvalues} associated with the operators $\Pkpm$ and the potential $r^{-\gamma}$ in the punctured ball. 
 
 The analogous problem for fully nonlinear uniformly elliptic operators has been recently studied in 
 \cite{BDL1}, where we considered in particular the eigenvalue problem 
$$\left\{ 
\begin{array} {cl}
\displaystyle {\cal M}^\pm  ( D^2 u) + \lambda_\gamma \frac{u}{ r^\gamma} = 0 & \quad \hbox{ in } B(0,1)\setminus \{0\}\\[2ex]
u=0 & \quad \hbox{ on } \partial B(0,1)
\end{array}
\right.
$$
Here, ${\cal M}^\pm$ are the Pucci's extremal operators (see \cite{CC}).   When $\gamma <2$, we proved that the principal eigenvalue is strictly positive and finite and  there  exist positive eigenfunctions which  can be  extended  by continuity to the whole ball. Moreover, the eigenfunctions  are  showed to be of class ${ \cal C}^1$ when $\gamma <1$, and Hoelder continuous  with exponent $2-\gamma$ if $\gamma \geq 1$. 
 In the case $\gamma =2$, we showed the non-existence of bounded super-solutions  and we  provided the explicit expression both of  eigenfunctions and the eigenvalues
        $\lambda_2$, thus extending to the fully nonlinear setting the explicit value of the best constant in Hardy's inequality. Finally, for $\gamma >2$, we proved that $\lambda_\gamma=0$.
  
 For problem \eqref{eigen} considered here, we obtain similar results in case of the operator $\Pkp$ and radically different results for the operator $\Pkm$. 
 
 For both operators, as in \cite{BGI18}, we consider the Berestycki,  Nirenberg and Varadhan constants $\lambda_\gamma$, see definitions \eqref{eigenbnv+} and \eqref{eigenbnv-}, as  good candidates to be the principal eigenvalues.
 
 In case of operator $\Pkp$, we attack the problem by first observing that if a positive radial eigenfunction $u$ related to a positive eigenvalue exists, then  it must satisfy $u''\geq \frac{u'}{r}$, so that $u$ actually is a radial eigenfunction for Laplace operator in dimension $k$. This property suggests to apply, in order to prove the existence of eigenfunctions, the same strategy used in the uniformly elliptic case. For $\gamma<2$ and $k\geq2$, after proving the validity of the maximum principle "below" $\lambda_\gamma$, the strict positivity of $\lambda_\gamma$ and the existence of positive eigenfunctions are obtained by a standard approximation procedure from below. When $k\geq 3$, we can let $\gamma\to 2$ and, by using the stability properties of the principal eigenvalues as well as their variational formulation in case of linear operators, we obtain the explicit expressions
 $$
 \lambda_2=\left(\frac{k-2}{2}\right)^2\, ,\qquad u(r)=\frac{-\ln r}{r^{\frac{k-2}{2}}}\, ,
 $$
 which are nothing but the well known expressions respectively of the first eigenvalue and of the eigenfunction  for Laplace operator in dimension $k$.
 
 We further show that $\lambda_\gamma=0$ for $\gamma>2$ and, in the separately considered case
  $k=1$, we prove that $\lambda_\gamma>0$ for $\gamma<1$ and $\lambda_\gamma=0$ for $\gamma\geq 1$.
  
 On the other hand,  when considering problem \eqref{eigen} in case of operator $\Pkm$, we see how crucial the degeneracy of ellipticity can be. Indeed, for any $\gamma\geq 0$, when $1\leq k<N$, we show \emph{for  all} $\mu>0$  the existence of explicit positive radial solutions of
 \begin{equation}\label{Pmmu}
  \Pkm(D^2v)+\mu \frac{v}{r^\gamma}=0\qquad \hbox{ in } B(0,1)\setminus \{0\}\, ,
 \end{equation}
proving that 
$$\lambda_\gamma=+\infty\qquad \hbox{ for all } \gamma\geq 0\, .$$
For equation \eqref{Pmmu} we also provide several sufficient conditions for the validity of the maximum principle, see Proposition \ref{maxpPk-}.
\medskip

Next, let us consider the existence of radial positive solutions of equation \eqref{super}  and their asymptotic behavior near the origin. The analogous problem for uniformly elliptic operators has been recently studied in \cite{BDL2}, where we considered solutions related to Pucci's operators, namely positive radial functions $u$ satisfying
$$
 {\cal M}_k^\pm  ( D^2 u) + \mu \frac{u}{r^2} = u^p \qquad \hbox{ in } B(0,1)\setminus \{0\}\, .
 $$
 The results obtained in \cite{BDL2} extend to the fully nonlinear framework the results obtained  by Cirstea \cite{Ci} and Cirstea and Du \cite{CDu2} in the case of the laplacian, that is for the equation
 \begin{equation}\label{semilinear}
 \Delta u + \mu \frac{u}{r^2} = u^p \qquad \hbox{ in } B(0,1)\setminus \{0\}\, .
 \end{equation}
 Let us recall that, when $0<\mu<\left(\frac{N-2}{2}\right)^2$, there are three natural growth exponents associated with solutions of the above equation. 
 The first one is the scaling invariance exponent $\frac{2}{p-1}$, and the other two are given by the roots of the equation
 $$
 x^2-(N-2)x+\mu=0\, ,
 $$
 namely 
$$
\tau^{\pm} \, := \frac{N-2}{2} \pm \sqrt{\left(\frac{N-2}{2}\right)^2-\mu}\, .
$$
 The functions $\frac{1}{r^{\tau^\pm}}$ are solutions of the homogeneous equation
 $$
 \Delta u + \mu \frac{u}{r^2} = 0 \qquad \hbox{ in } B(0,1)\setminus \{0\}\, ,
 $$
and they appear as asymptotic growth exponents of solutions $u$ of equation \eqref{semilinear} when
 the superlinear term $u^p$ becomes   negligible with respect to the singular term $\mu \frac{u}{r^2}$, or, in other words, when solutions are asymptotic to the solutions of the linearized equation at zero.

On the other hand, when either $\frac{2}{p-1}<\tau^-$ or $\frac{2}{p-1}>\tau^+$, equation \eqref{semilinear} has the exact solution
$$
u(r)=\frac{C_p}{r^{\frac{2}{p-1}}}\, ,\qquad C_p= \left[ \left(\frac{2}{p-1}-\tau^-\right) \left(\frac{2}{p-1}-\tau^+\right)\right]^{\frac{1}{p-1}}\, .
$$
 The asymptotic analysis of solutions of equation \eqref{semilinear} reveals the following classification, see \cite{BDL2,Ci, CDu1}: 
 \begin{itemize}
 \item[-] if $\frac{2}{p-1}> \tau^+$, then solutions $u$ satisfy either $u(r)r^{\frac{2}{p-1}}\to C_p$, or $u(r)r^{\tau^-}\to c_1$, or $u(r)r^{\tau^+}\to c_2$ as $r\to 0$, for some positive constants $c_1$ and $c_2$; 
 \item[-] if $\tau^-<\frac{2}{p-1}\leq \tau^+$, then any solution $u$ satisfies $u(r)r^{\tau^-}\to c_1$ as $r\to 0$, for some $c_1>0$ ; 
 \item[-] if $\frac{2}{p-1}=\tau^-$, then any solution $u$ satisfies $u(r)r^{\tau^-}(-\ln r)^{\frac{\tau^-}{2}}\to c_3$ as $r\to 0$, for some $c_3>0$ ; 
 \item[-] if $\frac{2}{p-1}<\tau^-$, then any solution $u$ satisfies $u(r)r^{\frac{2}{p-1}}\to C_p$ as $r\to 0$.
 \end{itemize}
 When considering solutions of equation \eqref{super} with the truncated laplacians, once again the pictures of the results for $\Pkp$ and $\Pkm$ are different. 
 
 For operator $\Pkp$, we assume   $k\geq 3$ and  we obtain exactly the same classification result as above with the exponents $\tau^\pm$ now defined with $N$ replaced by $k$. This is due to the fact that our analysis in the end shows that solutions of equation \eqref{super} with $\Pkp$ satisfy $u''\geq 0\geq u'$, so that they are indeed solutions of the semilnear equation \eqref{semilinear} in dimension $k$. Let us emphasize in particular that, for $\frac{2}{p-1}>\tau^+$, multiple asymptotic behaviors actually occur for solutions of equation \eqref{super} with $\Pkp$.
  
This non uniqueness phenomenon never occurs for operator $\Pkm$. Indeed, our analysis for equation \eqref{super} with $\Pkm$ shows that solutions always satisfy the associated ordinary first order equation, and then we deduce the  explicit expression of solutions, at least for $r>0$ sufficiently small. More precisely, assuming that $2\leq k\leq N-1$, we prove that for $r>0$ small enough, one has
\begin{itemize}
\item[-] if $\frac{2}{p-1}\neq \frac{\mu}{k}$, then  any solution $u$ is of the form
$$
u(r) =\frac{1}{\left( c\, r^{\frac{\mu (p-1)}{k}}- \frac{r^2}{\frac{2k}{p-1}-\mu}\right)^{\frac{1}{p-1}}}
$$
for some $c\in \R$;
\item[-] if $\frac{2}{p-1}= \frac{\mu}{k}$, then  any solution $u$ is of the form
$$
u(r)= \frac{1}{r^{\frac{2}{p-1}}\left(c+\frac{(p-1)}{k}(-\ln r)\right)^{\frac{1}{p-1}}}
$$
for some $c\in \R$.
\end{itemize}
Hence, in this case we obtain the following classification of the possible asymptotic behavior as $r\to 0$ for solutions:
\begin{itemize}
\item[-] if $\frac{2}{p-1}>\frac{\mu}{k}$, then $u(r)r^{\frac{\mu}{k}}\to c_1>0$;
\item[-] if $\frac{2}{p-1}=\frac{\mu}{k}$, then $u(r)r^{\frac{\mu}{k}}(-\ln r)^{\frac{\mu}{2k}}\to c_2>0$;
\item[-] if $\frac{2}{p-1}<\frac{\mu}{k}$, then $u(r)r^{\frac{2}{p-1}}\to c_3>0$.
\end{itemize}

 In conclusion, we can say that  the results of the present paper in a sense complement those contained in \cite{BDL1, BDL2} with respect to the uniform ellipticity assumption. Indeed, on the one hand the results we obtain here for operator $\Pkp$ show that the uniform ellipticity assumption is not necessary, and in some cases  partial analogous conclusions can still be obtained. On the other hand, the present results for operator $\Pkm$ show that a kind of  control on the ellipticity of the operator is in general necessary, since in case of  strong degeneracy   for the operator the results for both problems \eqref{eigen} and \eqref{super} can be radically different.


 \section{Preliminaries} 
 We begin by stating some technical results that will be used many times in the proofs of this paper. Let us recall that  by $B(0,1)$ we denote the unit ball of $\R^N$ centered at zero.
  
      \begin{lemme}\label{lem1}
Let $u\in { \cal C}^2(B(0,1)\setminus \{0\})$ be a radial function. Then:
\begin{itemize}
\item[{\rm (i)}] if $\Pkp (D^2u)\leq 0$  in $B(0,1)\setminus \{0\}$ with $k\leq N-1$, then $u^\prime\leq 0$ in $B(0,1)\setminus \{0\}$;
\item[{\rm (ii)}] if $\Pkm (D^2u)\leq 0$  in $B(0,1)\setminus \{0\}$, then $u^\prime$ has constant sign in a neighborhood of zero. Moreover, if $k\geq2$ and $u$ is bounded from below, then $u^\prime\leq 0$  in a neighborhood of zero and,  if $\Pkm(D^2u)<0$ in $B(0,1)\setminus \{0\}$, then
 $u^\prime\leq 0$ in $B(0,1)\setminus \{0\}$.
\end{itemize}
 \end{lemme}

 \begin{proof}
 For the proof of (i), it is enough to observe that, when  $k\leq N-1$, one has
 $$
 \Pkp(D^2u)\geq k\, \frac{u'}{r}\, .
 $$
 In order to prove (ii), we remark that, if $\Pkm(D^2u)\leq 0$, then  there cannot exist $0<s<t<1$ such that $u'(s)=u'(t)=0$ and $u'(r)>0$ for $s<r<t$. Indeed, in such a case, there would exists $\sigma\in (s,t)$ such that $u''(\sigma)>0$ and this would imply $\Pkm(D^2u(\sigma))>0$, against the assumption. This property implies that $u'(r)$ cannot change sign indefinitely  for $r$ in a neighborhood of zero.
 
Next, let us assume that $k\geq 2$, $u$ is bounded from below and, by contradiction, $u'(r)>0$ for $r$ small enough. Then, for $r$ in a right neighborhood of zero, it must be $u''(r)<0$ and
 $$
0\geq  \Pkm(D^2u)= u'' + (k-1)\, \frac{u'}{r}= \frac{(r^{k-1}u'(r))'}{r^{k-1}}\, .
 $$
Hence, the function $u^\prime(r) r^{k-1}$ is monotone non increasing and strictly positive for $r$ small enough.
This implies that, for some positive constant $c>0$ and for $r$ sufficiently small, one has
$$
u(r)\leq \left\{ \begin{array} {ll}
\displaystyle -\frac{c}{r^{k-2}} & \hbox{ if } k>2\\[1ex]
c \, \ln r & \hbox{ if } k=2
\end{array}\right.
$$
Letting $r\to 0$, we obtain a contradiction to the boundedness from below of $u$.

Hence, we proved that $u'(r)\leq 0$ for $r>0$ small enough. Assume now that, in addition,  $\Pkm(D^2u)<0$ in $B(0,1)\setminus \{0\}$. If, by contradiction, there exists $r>0$ such that $u'(r)>0$, then there exists also $s\in (0,1)$ such that $u'(s)=0$ and $u'(r)>0$ for $r>s$ close to $s$. But, then, $u''(s)\geq 0$, which implies $\Pkm(D^2u)(s)\geq 0$, in contrast with the strict inequality satisfied by $\Pkm(D^2u)$.

 \end{proof}

 \begin{rema} \label{remsign} {\rm 
 By using the relationship
 $$\Pkp(D^2(-u))=-\Pkm(D^2u)\, ,
 $$
 we immediately deduce from Lemma \ref{lem1} the symmetric statements for sub solutions, namely
 \begin{itemize}
\item[{\rm (i)}] if $\Pkm (D^2u)\geq 0$  in $B(0,1)\setminus \{0\}$ with $k\leq N-1$, then $u^\prime\geq 0$ in $B(0,1)\setminus \{0\}$;
\item[{\rm (ii)}] if $\Pkp (D^2u)\geq 0$  in $B(0,1)\setminus \{0\}$, then $u^\prime$ has constant sign in a neighborhood of zero. Moreover, if $k\geq2$ and $u$ is bounded from above, then $u^\prime\geq 0$  in a neighborhood of zero and,  if $\Pkp(D^2u)>0$ in $B(0,1)\setminus \{0\}$, then
 $u^\prime\geq 0$ in $B(0,1)\setminus \{0\}$.
\end{itemize}}
 \end{rema}

Our next result is  more specific: we focus on solutions with a monotone right hand side.
\begin{lemme}\label{lem1.1}
Let $g\in C^1(B(0,1)\setminus \{0\})$ be a radial function,  satisfying $g(r)< 0$ and $g^\prime(r)\geq0$ for $0<r<1$. Then:
\begin{itemize}
\item[{\rm (i)}] if  $u\in C^2(B(0,1)\setminus \{0\})$ is a radial bounded from below solution of
$$
\Pkp(D^2u)=g(r) \qquad \hbox{ in } B(0,1)\setminus \{0\}\, ,
$$
then $u^\prime \leq 0$ and   $u^{ \prime \prime}- { u^\prime \over r}\geq 0$ in $B(0,1)$;
\item[{\rm (ii)}] if  $u\in C^2(B(0,1)\setminus \{0\})$ is a radial bounded from below solution of
$$
\Pkm(D^2u)=g(r) \qquad \hbox{ in } B(0,1)\setminus \{0\}\, ,
$$
with $k\geq 2$, then $u^\prime \leq 0$ and   $u^{ \prime \prime}- { u^\prime \over r}$ has constant sign in a neighborhood of zero.
\end{itemize}
\end{lemme}

\begin{proof} (i) The inequality $u'\leq 0$ follows from Lemma \ref{lem1} (i) in case $k\leq N-1$, and from Lemma \ref{lem1} (ii) in case $k=N$.
Next, let us assume by contradiction that $u^{ \prime \prime}- { u^\prime \over r}<0$ in some interval contained in $(0,1)$. Then, from the equation   we get that   in such an  interval $u$ satisfies
 $k\frac{u^\prime(r)}{r}=g(r)$ and, by differentiating,  we obtain the contradiction
 $$0>\frac{k}{r} \left(u^{ \prime \prime}- { u^\prime \over r}\right)=g'(r)\geq 0\, .$$
 (ii) The inequality $u'\leq 0$ follows from Lemma \ref{lem1} (ii). Next, assuming by contradiction that $u''-\frac{u'}{r}$ changes sign infinitely many times as $r\to 0$, there exists an interval $(s,t)\subset (0,1)$ such that $u''(r)-\frac{u'(r)}{r}<0$ for $r\in (s,t)$ and $u''(s)-\frac{u'(s)}{s}=u''(t)-\frac{u'(t)}{t}=0$. Then, in the interval $(s,t)$, $u$ satisfies
 $$
 u''-\frac{u'}{r}=g-k \frac{u'}{r} \, ,
 $$
 and then, by differentiation,
 $$
 \left(u''-\frac{u'}{r}\right)' = g'(r)- \frac{k}{r} \left( u''-\frac{u'}{r}\right)>0\, ,
 $$
 in contradiction with the boundary conditions $u''(s)-\frac{u'(s)}{s}=u''(t)-\frac{u'(t)}{t}=0$.

\end{proof}

As an immediate consequence of Lemma \ref{lem1} and of Lemma \ref{lem1.1}, we deduce the following result for solutions of the eigenvalue problem for operator $\Pkp$.

\begin{cor}\label{corollary}
For $\gamma\geq 0$ and $\mu > 0$, let  $u\in C^2(B(0,1)\setminus \{0\})$ be a radial positive solution of
$$
\Pkp(D^2u)+\mu \frac{u}{r^\gamma}=0 \qquad \hbox{ in } B(0,1)\setminus \{0\}\, .
$$
Then, $u^\prime \leq 0$ and   $u^{ \prime \prime}- { u^\prime \over r}\geq 0$ in $B(0,1)\setminus \{0\}$.
\end{cor}

Our last preliminary result is a comparison principle for radial solutions in the punctured ball satisfying a suitable growth condition around the singularity. The proof, based merely on the ellipticity property of the principal part and on the superlinear character of the absorbing zero order term, is given in \cite{BDL2} for Pucci's extremal operator $\mathcal{M}^+$, but it applies exactly in the same way to both operators ${\cal P}_k^\pm$.

\begin{theo}\label{comparison}
Let  $\mu>0$, $p>1$, $r_0>0$ and let $u$ and $v$ be two  positive radial  functions in $C^2\left(B_{r_0}\setminus \{0\}\right)$ satisfying
$$
{\cal P}_k^\pm(D^2u)+\mu \frac{u}{r^2} -u^p\geq 0 \geq {\cal P}_k^\pm(D^2v) +\mu \frac{v}{r^2} -v^p\quad \hbox{ in } B_{r_0}\setminus \{0\}\, .
$$
Assume that $u\, ,\ v \in C\left( \overline{B_{r_0}}\setminus \{0\}\right)$ and that
 there exist positive constants $c_1, c_2$ such that 
 $$ c_1 r^{-2\over p-1}\leq  v(r)\quad \hbox{ and }\quad   u(r)\, , v(r) \leq c_2 r^{-2\over p-1}\quad \hbox{  for }\  0<r\leq r_0\, .$$
 If $u(r_0)\leq v(r_0)$, then $u\leq v$ in $\overline{B_{r_0}}\setminus \{0\}$.
             \end{theo}


 \section{Principal eigenvalues and  eigenfunctions}\label{eigenvalues}
 \subsection{The ${\cal P}_k^+ $ case, with $k\geq 2$}
 Let $\gamma>0$. Following the approach of \cite{BNV},  we introduce
  \begin{equation}\label{eigenbnv+}
  \begin{array}{c}
 \displaystyle   \bar \lambda_\gamma \, : = \sup \left\{ \mu\in \R\, : \,  \exists \, v \in { \cal C}^2 (B(0,1)\setminus\{0\}) \hbox{ s.t. } v>0\, , v \hbox{ is radial, }\right.\\[1ex]
 \qquad\qquad\qquad\qquad  \left. {\cal P}_k^+ ( D^2 v) + \mu v r^{-\gamma} \leq 0 \hbox{ in } B(0,1)\setminus\{0\}\right\}\, ,
 \end{array}
 \end{equation}
 which will be referred to  as the radial principal eigenvalue for ${\cal P}_k^+$ and the potential $r^{-\gamma}$ in $B(0,1)\setminus \{0\}$. 

We immediately observe that $\bar \lambda_\gamma\geq 0$, as it follows  by using constant functions as super-solutions $v$. More than that, in the case $\gamma< 2$, an easy computation shows that the function
 $v(r)=1-r^\beta$
 satisfies, for $0 < \beta \leq 2-\gamma$,
 $$
 {\cal P}_k^+(D^2v)\leq -c \frac{v}{r^\gamma}
 $$
 for some $c>0$ depending on $\beta$ and $k$. Hence, we deduce that $\bar \lambda_\gamma >0$ at least for $\gamma < 2$.
  
 The aim of the following results is to provide sufficient conditions ensuring that $\bar \lambda_\gamma$ actually is a positive  eigenvalue associated with ${\cal P}_k^+$ and the potential $r^{-\gamma}$, meaning that there exist positive  solutions of the Dirichlet problem
 \begin{equation}\label{exieigen}
 \left\{
 \begin{array}{cl}
 \displaystyle {\cal P}_k^+(D^2u) + \bar \lambda_\gamma \frac{u}{r^\gamma}=0  & \hbox{ in } B(0,1)\setminus \{0\}\\[2ex]
 \displaystyle u=0  &  \hbox{ on } \partial B(0,1)
 \end{array} \right. 
 \end{equation}
Let us start by proving a maximum principle \lq\lq below" $\bar\lambda_\gamma$.
    \begin{theo}\label{maxp1}
 Let  $u\in { \cal C}^2 ( B(0,1)\setminus\{0\})\cap {\cal C}(\overline{B(0,1)})$ be a radial function satisfying
     $$ {\cal P}_k^+ ( D^2 u) + \mu u r^{-\gamma} \geq 0 \qquad \hbox{ in } B(0,1)\setminus \{0\}$$
     for some $\mu < \bar \lambda_\gamma$. If $u(1)\leq 0$, then
 $$u(r)\leq 0 \qquad \hbox{ for } 0\leq r\leq 1\, .$$ 
\end{theo}
 \begin{proof} The proof proceeds as the proof of Theorem 2.7 in \cite{BDL1}. We sketch the steps for the reader's convenience.

Since $\mu < \bar \lambda_\gamma$, by definition of $\bar \lambda_\gamma$ we can select $\mu'$, with $\mu<\mu'\leq \bar \lambda_\gamma$, for which there exists a positive radial function $v \in { \cal C}^2 (B(0,1)\setminus\{0\})$ satisfying \
 $$ {\cal P}_k^+ ( D^2 v) + \mu' v r^{-\gamma} \leq 0 \hbox{ in } B(0,1)\setminus\{0\}\, .$$
 Without loss of generality, we can assume that $\mu'\geq 0$, $v(r)$ is continuous up to $r=1$ and  $v(r)>0$ for $0<r\leq 1$. Moreover, by Lemma \ref{lem1}, $v$ is monotone non increasing for $r$ close to zero, so that there exists the limit $\lim_{r\to 0} v(r) = : v(0)\in (0,+\infty]$.

Let us consider the quotient function $\frac{u}{v}$ and its supremum value on $B(0,1)\setminus \{0\}$. Arguing by contradiction, let us assume that the supremum is positive. Up  to a normalization, we can assume that 
$$
\sup_{B(0,1)\setminus \{0\}}\frac{u}{v}=1
$$
Let $\bar r\in [0,1)$ be any limit point of a maximizing sequence. If $\bar r>0$, then $u(\bar r)=v(\bar r)$ and
  testing the differential inequalities satisfied by $u$ and $v$ at $\bar r$ yields, by ellipticity,
 $$
 \mu'  v(\bar r) {\bar r}^{-\gamma} \leq -\Pkp (D^2v(\bar r))\leq - \Pkp (D^2u(\bar r))\leq \mu u(\bar r) {\bar r}^{-\gamma}\, ,
 $$
 which is a contradiction to the inequality $\mu <\mu'$.  On the other hand, if $\bar r=0$,  then $v(0)<+\infty$, $u(0)= v(0)$ and $\mu^\prime v(0)-\mu u(0) >0$. Hence, by the sub-additivity of operator $\Pkp$ and by continuity, in a neighborhood of zero one has 
 $${ \cal P}_k^+ ( D^2 ( u-v)(r))  \geq \Pkp(D^2u) -\Pkp(D^2v)\geq \left(\mu' v(r)-\mu u(r)\right)r^{-\gamma}\geq {1\over 2} \left( \mu^\prime v(0)- \mu u(0)\right) r^{-\gamma}>0\, .$$
 By Remark \ref{remsign}, this implies that $(u-v)^\prime \geq 0$ in a neighborhood of zero, so that  $u-v$ cannot achieve its  maximum at zero. 
          
 \end{proof}     
 
  \begin{rema} \label{remcompbelow} {\rm 
  Let us explicitly observe that the assumption $k\geq 2$ is crucial for the validity of the maximum principle stated by Theorem \ref{maxp1}. Indeed, it fails for $k=1$, as shown by the function $u(r)=1-r\in C(\overline{B(0,1)})\cap C^2 (B(0,1)\setminus \{0\})$,which satisfies
  $$
  {\cal P}^+_1 (D^2u)=0 \qquad \hbox{ in } B(0,1)\setminus \{0\}
  $$ 
  and $u(1)=0$, $u>0$ in $B(0,1)$.
    }
  \end{rema}

 Thanks to the above maximum principle, we can prove the following existence result.

 \begin{prop}\label{exifmu}
Let  $b\geq 0$, $\gamma< 2$ and $0<\mu < \bar \lambda_\gamma$\, . Then, for any radial nondecreasing function $f\in  { \cal C}^1( B(0,1))$,   such that  $f<0$ and $f'\geq 0$ in $B(0,1)$,  there exists $u\in {\cal C}(\overline{B(0,1)})\cap {\cal C}^2(B(0,1)\setminus \{0\})$ solution of
$$
\left\{
\begin{array}{cl}
  \displaystyle {\cal P}_k^+(D^2u) + \mu \frac{u}{r^\gamma}=\frac{f}{r^\gamma}  & \hbox{ in } B(0,1)\setminus \{0\}\\[2ex]
 \displaystyle u=b  &  \hbox{ on } \partial B(0,1)
 \end{array} \right. 
 $$
Furthermore,   $u$ satisfies $u^\prime \leq0$ and $u^{ \prime \prime } \geq \frac{u^\prime}{r}$ in  $B(0,1)\setminus \{0\}$. 
           \end{prop}
 \begin{proof}  Let us inductively define the following sequence of functions: we set $u_0(r)\equiv b$ and, for $n\geq 1$,
 $$
 u_n(r) \, : = b -\int_r^1 \frac{1}{s^{k-1}}\int_0^s \left(f(t)-\mu\, u_{n-1}(t)\right)t^{k-1-\gamma} dt\, ds\, .
 $$
 The assumptions on $\gamma$,  $b$ and $f$ make $\{ u_n\}$ a well defined sequence of non negative continuous functions in $[0,1]$, of class ${\cal C}^2$ in $(0,1)$ and  satisfying
 $$
 u'_n(r)\leq 0\, ,\quad u''_{n}(r)-\frac{u'_n(r)}{r} =- \frac{k}{r^k} \int_0^r \left(f(t)-\mu\, u_{n-1}(t)\right)t^{k-1-\gamma} dt +\left(f(r)-\mu\, u_{n-1}(r)\right)r^{-\gamma} \geq 0\, .
 $$
 Hence, for every $n\geq 1$, $u_n\in {\cal C}(\overline{B(0,1)})\cap {\cal C}^2(B(0,1)\setminus \{0\})$ is a radial positive solution of
$$
\left\{
\begin{array}{cl}
  \displaystyle {\cal P}_k^+(D^2u_n) =\left( f-\mu\, u_{n-1}\right)\, r^{-\gamma}  & \hbox{ in } B(0,1)\setminus \{0\}\\[2ex]
 \displaystyle u_n=b  &  \hbox{ on } \partial B(0,1)
 \end{array} \right. 
 $$
Furthermore, a simple induction argument shows that $u_n\geq u_{n-1}$ in $[0,1]$ for all $n\geq 1$. Since $f<0$, this implies also that 
$u_n$ is not identically zero. 

We claim that $\{u_n$\} is  bounded in ${\cal C}(\overline{B(0,1)})$. Indeed,
if not, dividing by $\|u_{n}\|_\infty=u_n(0)\rightarrow +\infty$ and defining  $v_n = \frac{u_n}{\|u_n\|_\infty}$,  one gets that    $v_n$ is a ${ \cal C}^2( B(0,1)\setminus \{0\})$ radial function satisfying $v_n^\prime \leq 0$, $v_n^{ \prime \prime } \geq { v_n^\prime \over r}$ and
  $$
  \left\{
\begin{array}{cl}
  \displaystyle { \cal P}_k^+ ( D^2 v_{n} )  = \left(\frac{f}{\|u_n\|_\infty}-\mu \, v_{n-1}\frac{ \|u_{n-1}\|_\infty}{ \|u_n\|_\infty}\right) r^{-\gamma}& \hbox{ in } B(0,1)\setminus \{0\}\\[2ex]
\displaystyle v_n=\frac{b}{\|u_n\|_\infty}  &  \hbox{ on } \partial B(0,1)
 \end{array} \right. 
$$
Furthermore 
               $$ v^\prime_n \geq - \frac{1}{k-\gamma} \frac{\|f-\mu u_{n-1}\|_{\infty}}{\|u_n \|_\infty } r^{1-\gamma}\geq -C r^{1-\gamma}\, , $$ 
which implies that $\{v_n^\prime\}$ is locally uniformly bounded in $(0,1)$, and by the equation, so is $\{v_n^{ \prime \prime}\}$. We can let $n\to \infty$  in $B(0,1) \setminus \{0\}$
and we get that $v_n \rightarrow v$, $v_n^\prime \rightarrow v^\prime $  locally uniformly and, by the equation,  $v_n^{ \prime \prime} \rightarrow v^{ \prime \prime}$ as well. Furthermore $v$ is ${ \cal C}^2( B(0, 1)\setminus \{0\})$, $v^\prime \leq 0$, $v^{ \prime \prime } \geq { v^\prime \over r}$, and, possibly considering a subsequence such that there exists the limit $\lim_{n\rightarrow\infty}\frac{ \|u_{n-1}\|_\infty}{ \|u_n\|_\infty}=c\leq 1$, $v$ satisfies
              $$ { \cal P}_k^+ ( D^2 v) + c\mu v r^{-\gamma} = 0\, ,\ \ \  v(1) = 0\, .$$ 
By Theorem \ref{maxp1} we obtain $v\equiv 0$, which contradicts $\|v\|_\infty =1$.

Now, with $\{u_n\}$  bounded, we can repeat what was done with $\{v_n\}$ and we obtain that  
  $u_n\rightarrow u$, with $u$ as in  the statement by Lemma \ref{lem1.1}.
                
           \end{proof} 
           
 We are now ready to prove our first existence result for  eigenfunctions in the case $\gamma<2$.
         
 \begin{theo}\label{exief} Suppose that $\gamma<2$. Then, there exists $u>0$,   $u\in C(\overline{B(0, 1)})\cap C^2(B(0,1) \setminus \{0\})$ solution of \eqref{exieigen}. 
   Furthermore, $u$ satisfies $u^\prime \leq0$ and $u^{ \prime \prime } \geq \frac{u^\prime}{r}$ in  $B(0,1)\setminus \{0\}$. 
 \end{theo}
 \begin{proof}    
 Let us select a positive increasing sequence $\lambda_n\rightarrow \bar\lambda_\gamma$. By Proposition \ref{exifmu}, for every $n\geq 1$, we can further consider a positive radial solution $u_n\in {\cal C} ( \overline{B(0,1)}) \cap { \cal C}^2(B(0,1)\setminus\{0\})$ of 
 $$
\left\{
\begin{array}{cl}
  \displaystyle {\cal P}_k^+(D^2u_n) +\lambda_n \frac{u_n}{r^\gamma}= -\frac{1}{r^\gamma} & \hbox{ in } B(0,1)\setminus \{0\}\\[2ex]
 \displaystyle u_n=0  &  \hbox{ on } \partial B(0,1)
 \end{array} \right. 
 $$
More explicitly, the functions $u_n$ satisfy
$$
\left\{
\begin{array}{cl}
  \displaystyle
u_n''+(k-1)\frac{u_n'}{r}+\lambda_n \frac{u_n}{r^\gamma}= -\frac{1}{r^\gamma}  &  \hbox{ for } 0<r<1\, ,\\[2ex]
 \displaystyle u_n(1)=0 & 
 \end{array}\right.
 $$
We claim that the sequence  $\{ \|u_n\|_\infty\}$ is unbounded. Indeed, 
assuming, on the contrary, that $\{u_n\} $ is uniformly  bounded, we could extract  from $\{u_n\}_n$ a subsequence uniformly converging to a function $u$ satisfying
           $$ u^{ \prime \prime} + ( k-1) { u^\prime \over r} +\bar \lambda_\gamma {u \over r^\gamma}= {-1\over r^\gamma} \, .$$
 Moreover, $u$ would belong ${\cal C}(\overline{B(0,1)})\cap { \cal C}^2(B(0,1)\setminus\{0\})$ and satisfy $ u^{ \prime \prime}\geq  u^{ \prime }
r^{-1}$. Hence, $u$  would be a positive bounded solution of
$${ \cal P}_k^+ ( D^2 u) + \bar \lambda_\gamma \frac{u}{ r^{\gamma}} = -r^{-\gamma}\qquad \hbox{ in } B(0,1)\setminus \{0\}\, .$$
As a consequence, for $\epsilon>0$ sufficiently small, $u\in { \cal C}^2(B(0,1)\setminus\{0\})$ would be a positive radial function satisfying
 $${ \cal P}_k^+ ( D^2u) + (\bar\lambda_\gamma +\epsilon) \frac{u}{r^\gamma} \leq 0$$
in  contradiction with  the definition of $\bar\lambda_\gamma$.

 This proves  that the sequence $\{u_n\}_n$ is  not uniformly bounded. Then, by arguing on the normalized sequence $\{ v_n:= 
 \frac{u_n }{\|u_n\|_\infty}\}$, we can let $n\to \infty$ in order to obtain a limit function $u$ as in the statement.  

\end{proof}

Let us now make an observation relying on Corollary \ref{corollary}. Since any eigenfunction $u$ solving \eqref{exieigen} actually satisfies $u''\geq \frac{u'}{r}$, then $u$ is a solution of the ordinary differential equation
$$
u''+(k-1) \frac{u'}{r} +\bar \lambda_\gamma \frac{u}{r^\gamma}=0\qquad \hbox{ in } (0,1)\, .
$$
Therefore, the existence of $u$ can be proved also by following a variational approach for the above linear eigenvalue problem. The two approaches lead to the same result, as stated in the following proposition, where 
 we use the notation
 $$
{\cal  H}_0^{1,k}= \{ u\in W^{1,1}((0,1))\, :\ u(1)=0\, ,\  \int_0^1 (u^\prime)^2 r^{k-1} dr<+\infty \} 
 $$    
    \begin{prop}\label{variational}
   For any $\gamma <2$, one has 
   $$\bar \lambda_\gamma=  \lambda_{\gamma ,{\rm var}} \, : = \inf_{\{ u\in {\cal H}_0^{1,k} \, :\ \int _0^1 u^2 r^{k-1-\gamma}dr = 1\} } \int_0^1
    (u^\prime )^2 r^{k-1} dr \, .$$
      \end{prop}

 \begin{proof} 
 The existence of a minimum realizing the infimum defining $\bar \lambda_{\gamma, {\rm var}}$ is classical.  
 
 Therefore, there exists $u\in {\cal H}_0^{1,k}$, with $u>0$, satisfying in the distributional sense
 $$
 (u'r^{k-1})'=- \lambda_{\gamma, {\rm var}} u \, r^{k-1-\gamma}\, .$$
 Hence, $(u'r^{k-1})'$ is continuous in $(0,1)$ and this implies that $u$ is in ${\cal C}^2((0,1))$. Thus, $u$ is a classical solution of
 $$
 u^{ \prime \prime } + ( k-1) \frac{u^\prime}{ r} = - \lambda_{\gamma, {\rm var}} \frac{u}{ r^\gamma} \ \mbox{ in } (0,1)\, ,\quad u(1)=0\, .$$
By monotonicity, there exists the limit $\lim_{r\to 0} u'(r) r^{k-1}$, which must be lesser than or equal  to zero, otherwise $u(r)$ would become infinitely negative for $r\to 0$.  Thus, one has $u'(r)\leq 0$ for $r\in (0,1)$ and
  then $-\lambda_{\gamma, {\rm var}} u(r)r^{-\gamma}$ is  increasing. Integrating the above equation yields
 $$ u^\prime (r) r^{k-1} \leq -\lambda_{\gamma, {\rm var}} u(r) r^{-\gamma} \int_0^r s^{k-1}ds= -\frac{\lambda_{\gamma, {\rm var}}}{k}u(r) r^{k-\gamma}\, ,$$
           and then 
 $$ \frac{u'}{r}\leq \frac{1}{k} \left( u''+(k-1) \frac{u'}{r}\right)\, ,$$
 that is
 $$\frac{u'}{r}\leq u''\, .$$
 Hence, $u\in {\cal C}^2(B(0,1)\setminus \{0\})$ is a positive radial solution of
 ${ \cal P}_k^+ ( D^2 u) + \lambda_{\gamma, {\rm var}} ur^{-\gamma} =0$ and this implies, 
  by  definition, that   $\lambda_{\gamma, {\rm var}}\leq \bar \lambda_\gamma$. On the other hand, if $\lambda_{\gamma, {\rm var}}< \bar \lambda_\gamma$, then Theorem \ref{maxp1} gives $u\leq 0$, which is impossible.
           
    \end{proof} 

Thanks to the characterization given by Proposition \ref{variational},  we can give the explicit expression of the eigenvalue and of an eigenfunction in the case $\gamma=2$.

\begin{theo} Let $\gamma=2$ and $k\geq 3$. Then 
$$
 \bar \lambda_2 = \frac{(k-2)^2}{4}$$
 and the function
$$
u(r)= \frac{- \ln r}{r^{\frac{k-2}{2}}}
$$
is an explicit solution of problem \eqref{exieigen} with $\gamma=2$.
\end{theo}
\begin{proof}
As it is well known (see e.g. \cite{GAPA}), the value $\frac{(k-2)^2}{4}$ is nothing but the inverse of the best constant in Hardy's inequality in dimension $k$. In other words, one has
$$
\frac{(k-2)^2}{4}= \lambda_{2, {\rm var}}
$$
and the function $u(r)= \frac{- \ln r}{r^{\frac{k-2}{2}}}$ is the explicit infinite energy solution of the eigenvalue problem
$$
\left\{
\begin{array}{cl}
\displaystyle -\Delta u = \lambda_{2,{\rm var}} \frac{u}{r^2} & \hbox { in } B(0,1)\subset \R^k\\[2ex]
\displaystyle u=0 & \hbox { on } \partial B(0,1)
\end{array}
\right.
$$
As a function of one variable, $u$ belongs to ${\cal C}^2((0,1])$ and satisfies $u''\geq \frac{u'}{r}$, so that $u$ is a classical solution
in dimension $N$ of
$$
{ \cal P}_k^+ ( D^2 u) + \lambda_{2, {\rm var}} ur^{-2} =0\qquad \hbox{ in } B(0,1)\setminus \{0\}\, .
$$
Hence, by definition, we deduce 
$$\bar \lambda_2\geq \lambda_{2,{\rm var}}= \frac{(k-2)^2}{4}\, .$$
On the other hand, again by the definition of $\bar \lambda_\gamma$, one easily sees that  the value $\bar \lambda_\gamma$ is monotone non increasing with respect to $\gamma$, hence
$$
\bar \lambda_2 \leq \lim_{\gamma \to 2^-} \bar \lambda_\gamma\, .
$$
By applying Proposition \ref{variational} and by using the stability with respect to $\gamma$ of the variational eigenvalue, we then obtain
$$
\bar \lambda_2 \leq \lim_{\gamma \to 2^-} \bar \lambda_\gamma= \lim_{\gamma \to 2^-} \lambda_{\gamma, {\rm var}}=\lambda_{2, {\rm var}}=\frac{(k-2)^2}{4}\, ,
$$
which concludes the proof.
\end{proof}

We conclude this section by considering the case $\gamma >2$.

\begin{theo}\label{gamma>2}
If $\gamma>2$, then $\bar \lambda_\gamma=0$.
\end{theo}

\begin{proof}
By contradiction, let us assume that for some $\mu>0$ there exists a positive radial function $v\in {\cal C}^2(B(0,1)\setminus \{0\})$ satisfying 
$$
{\cal P}^+_k(D^2v)+ \mu \frac{v}{r^\gamma}\leq 0\qquad \hbox{ in } B(0,1)\setminus \{0\}\, .
$$
Then, by Lemma \ref{lem1},  we have $v'\leq 0$. Moreover, by definition of ${\cal P}^+_k$, we have also that $v$ satisfies 
$$
v''+(k-1) \frac{v'}{r} +\mu \frac{v}{r^\gamma}\leq 0\qquad \hbox{ for } 0<r<1\, .
$$
Hence, we can apply e.g. the same proof of Theorem 1.3 in \cite{BDL1} in order to reach a contradiction.

\end{proof}

 \subsection{The ${\cal P}_1^+ $ case}   
 
 We consider here the eigenvalue problem for operator ${\cal P}^+_1$, which has to be studied separately because of the lack of validity of the maximum principle given by Theorem \ref{maxp1}. 
 
 Nonetheless, for $\gamma<1$, we still have a variational characterization of $\bar \lambda_\gamma$ analogous to Proposition \ref{variational}. In this case, we set
 $$
 \lambda_{\gamma, {\rm var}} \, := \inf_{\{ v\in H^1((0,1))\, : v(1)=0\, ,\ \int_0^1 v^2(r) r^{-\gamma}dr=1\}} \int_0^1 (v'(r))^2dr\, .
 $$
 Let us observe that, for $\gamma<1$, one has $\lambda_{\gamma, {\rm var}}>0$, as it is easily follows by H\"older inequality. 
 
 Indeed, for any $v\in C^1([0,1])$ satisfying $v(1)=0$, one has
 $$
 |v(r)|=\left| \int_r^1 v'(s)\, ds\right| \leq \left( \int_r^1 (v'(s))^2\, ds \right)^{1/2} (1-r)^{1/2}\, ,
 $$
 which yields
 $$
 \int_0^1 v^2(r)r^{-\gamma}dr\leq \left(\int_0^1 (v'(r))^2dr\right)\, \int_0^1 (1-r)r^{-\gamma}dr= \frac{1}{(1-\gamma)(2-\gamma)} \int_0^1 (v'(r))^2dr\, .
 $$
 Hence, we have
 $$
 \lambda_{\gamma, {\rm var}} \geq (1-\gamma)(2-\gamma)\, .
 $$

   \begin{theo}\label{nonexik=1} 
   
   \begin{itemize}
\item[(i)]   If  $\gamma <1$, then $\bar \lambda_\gamma =\lambda_{\gamma, {\rm var}}>0$ and problem \eqref{exieigen}  with $k=1$ has a positive radial solution belonging to ${\cal C}^2(B(0,1)\setminus \{0\})\cap {\cal C}^1(\overline{B(0,1)})$.
\item[(ii)] If $\gamma \geq 1$, then $\bar \lambda_\gamma=0$.
\end{itemize}
    \end{theo}
 \begin{proof}
{\sl (i)}  By classical methods in the calculus of variations, the infimum defining $\lambda_{\gamma, {\rm var}}$ is a minimum. Thus, there exists $u\in H^1((0,1))$, with $u>0$ in $(0,1)$, satisfying in the distributional sense
 $$
 u''= -\lambda_{\gamma, {\rm var}} u\, r^{-\gamma}\qquad \hbox{ in } (0,1)\, ,
 $$
 and the boundary conditions
 $$
 u'(0)=0\, ,\quad u(1)=0\, .
 $$
 As a consequence,  $u\in {\cal C}^2([0,1])$  satisfies $u''\, ,\ u'\leq 0$ in $[0,1]$. Moreover, by monotonicity, one has
 $$
 u'(r)=\int_0^r u''(s)\, ds =- \lambda_{\gamma, {\rm var}} \int_0^r u(s) s^{-\gamma}ds\leq - \lambda_{\gamma, {\rm var}} u(r) r^{1-\gamma} = u''(r) r\, .
 $$
 Therefore, $u\in {\cal C}^2(B(0,1)\setminus \{0\})\cap {\cal C}^1(\overline{B(0,1)})$ is a positive radial solution of 
 $$
 \left\{ \begin{array}{cl}
 \displaystyle {\cal P}^+_1 (D^2u) + \lambda_{\gamma, {\rm var}} u r^{-\gamma}=0 & \hbox{ in } B(0,1)\setminus \{0\}\\[2ex]
 \displaystyle u=0 & \hbox{ on } \partial B(0,1)
 \end{array}\right.
 $$
 and this implies, by definition, that $\bar \lambda_\gamma\geq \lambda_{\gamma {\rm var}}$.
 
 Arguing by contradiction, let us suppose  that $\bar \lambda_\gamma> \lambda_{\gamma {\rm var}}$. Then, there exist $\mu >\lambda_{\gamma {\rm var}}$ and $v\in {\cal C}^2(B(0,1)\setminus \{0\})$ positive and radial such that
 $$
 {\cal P}^+_1 (D^2v) + \mu v r^{-\gamma}\leq 0  \qquad \hbox{ in } B(0,1)\setminus \{0\}\, .
 $$
 Without loss of generality we can assume that $v(r)>0$ for $r\in (0,1]$. Moreover, since $ {\cal P}^+_1 (D^2v)=\max \left\{ v''\, ,\ \frac{v'}{r}\right\}$, we have that $v$ satisfy both the inequalities
 $$
 v'' + \mu v r^{-\gamma}\leq 0\, ,\quad v'+ \mu v r^{1-\gamma}\leq 0\qquad \hbox{ for } 0<r<1\, .
 $$
 In particular, we have that $v'\, ,\ v''\leq 0$ and there exists the limit $\lim_{r\to 0} v(r) =:\, v(0)\in (0,+\infty]$. We can now argue as in the proof of Theorem \ref{maxp1} in order to reach a contradiction. For, up to a normalization, we can assume that
 $$
 \sup_{B(0,1)\setminus \{0\}} \frac{u}{v}=1\, .
 $$
Let $\bar r\in [0,1)$ be any limit point of a maximizing sequence.  If $\bar r>0$, then $u(\bar r)=v(\bar r)$ and $u''(\bar r)\leq v''(\bar r)$, so that
$$
 \mu  v(\bar r) {\bar r}^{-\gamma} \leq -v''(\bar r))\leq - u''(\bar r))= \lambda_{\gamma,{\rm var}} u(\bar r) {\bar r}^{-\gamma}\, ,
 $$
 which gives the  contradiction  $\mu \leq \lambda_{\gamma,{\rm var}}$.  If $\bar r=0$,  then $v(0)<+\infty$, $u(0)= v(0)$,  and $\mu  v(0)-\lambda_{\gamma{\rm var}} u(0) >0$. Hence, by continuity, in a neighborhood of zero one has 
 $$(u-v)''(r)\geq \left(\mu v(r)-\lambda_{\gamma{\rm var}} u(r)\right)r^{-\gamma}\geq {1\over 2} \left( \mu  v(0)- \lambda_{\gamma, {\rm var}} u(0)\right) r^{-\gamma}>0\, .$$
Since $(u-v)'(0)\geq 0$, this implies that $(u-v)^\prime(r) \geq 0$ in a neighborhood of zero, so that  $u-v$ cannot achieve its  maximum at zero. The   contradictions reached in both cases prove that $\bar \lambda_\gamma= \lambda_{\gamma {\rm var}}$ and the variational solution $u$ actually is a smooth solution of problem \eqref{exieigen} with $k=1$.
\medskip

{\sl (ii)} By contradiction, let us assume that there exist  $\mu >0$ and $v\in {\cal C}^2(B(0,1)\setminus \{0\})$ positive and radial such that
 $$
 {\cal P}^+_1 (D^2v) + \mu v r^{-\gamma}\leq 0  \qquad \hbox{ in } B(0,1)\setminus \{0\}\, .
 $$
Then, as before, we have that $v$ in particular satisfies $v', v''\leq 0$ and
$$
v''\leq - \mu v r^{-\gamma}\qquad \hbox{ for } 0<r<1\, .
$$
Hence, for $r>0$ small, one has 
$v^{ \prime \prime } \leq -c r^{-\gamma}$ for some $c>0$. This implies that, for $r>0$ sufficiently small and for some $c_1>0$, one has
$$
v'(r)\leq \left\{ \begin{array}{ll}
-c_1 r^{1-\gamma} & \hbox{ if } \gamma>1\\[2ex]
c_1 \ln r & \hbox{ if } \gamma=1
\end{array} \right.
$$
In both cases,  this yields  $\lim_{r\to 0} v'(r)=-\infty$, which is a contradiction to $v''\leq 0$.

 \end{proof}

 \subsection{The ${\cal P}_k^- $ case}
  As before, we introduce 
  \begin{equation}\label{eigenbnv-}\begin{array}{c}
 \displaystyle   \underline{\lambda}_\gamma \, : = \sup \left\{ \mu\in \R\, : \,  \exists \, v \in { \cal C}^2 (B(0,1)\setminus\{0\}) \hbox{ s.t. } v>0\, , v \hbox{ is radial, }\right.\\[1ex]
 \qquad\qquad\qquad\qquad  \left. {\cal P}_k^-( D^2 v) + \mu v r^{-\gamma} \leq 0 \hbox{ in } B(0,1)\setminus\{0\}\right\}
 \end{array}
 \end{equation}
 but now  the results deeply differ from   the 
${\cal P}_k^+ $ case.
 \begin{theo} \label{Pk-}
 For any $\gamma \geq 0$   and any  $k<N$, one has $\ulambda=+\infty$. 
  \end{theo}
  
  \begin{proof}
  Let us first consider the case $\gamma \neq 2$. 

For $\mu >0$, let $v(r) = e^{-\frac{\mu}{k(2-\gamma)} r^{2-\gamma}}$. Then 
$v$ satisfies $$ {v^\prime \over r} = -{ \mu \over k} vr^{-\gamma}$$ and 
 $$ v^{ \prime \prime} -{v^\prime \over r} = \frac{\mu}{k} \gamma  v r^{-\gamma} + \frac{\mu^2}{k^2} v r^{2-2\gamma} \geq 0\, .$$
Since $k<N$, then we have
$$ {\cal P}_k^-( D^2 v)= k \frac{v'}{r}=- \mu v r^{-\gamma}\qquad \hbox{ in } B(0,1)\setminus \{0\}\, ,$$
 which implies that $\ulambda \geq \mu$. By the arbitrariness of $\mu>0$, one gets the desired result.  
     
For $\gamma = 2$, let us consider the function  $v(r)=  r^{-\mu \over k}$, which satisfies $v^\prime \leq 0\leq v^{ \prime \prime }$  and 
 $$\frac{v^\prime}{ r} = -\frac{\mu}{ k} v r^{-2}\, .$$
Hence,  $v$ is a positive solution in $B(0,1)\setminus \{0\}$ and, as before, we obtain $\ulambda \geq \mu$ for any $\mu>0$.

 \end{proof} 
 
 The explicit solutions exhibited in the above proof can be used as barrier functions in order  to  prove the following  maximum principle.
  \begin{prop} \label{maxpPk-}
For  $\mu>0$, let $u\in C^2(B(0,1)\setminus \{0\})\cap C(\overline{B(0,1)} \setminus \{0\})$ be a radial solution of 
       $$ { \cal P}_k^- ( D^2 u(r)) + \mu u(r) r^{-\gamma} \geq 0\qquad \hbox{ in } B(0,1)\setminus \{0\}\, ,$$ 
satisfying  $u(1)\leq 0$.
In any of the following cases
\begin{enumerate}       
     \item $\gamma <2$, $k\geq 2$  and $u$ is bounded
\item  $\gamma = 2$ and there exists some $\tau >0$ such that $u(r) r^\tau$ is bounded
        \item $\gamma >2$ and there exists some constant $c>0$ such  that
         $u(r) e^{ -c r^{2-\gamma}} $ is bounded
\end{enumerate}         
one has   $u\leq 0$ in $\overline{B(0,1)} \setminus \{0\}$. 
\end{prop}
 \begin{proof}
 
\emph{Case 1.}   For $\mu^\prime > \mu$, let us consider the function  $v(r)= e^{-\frac{\mu'}{k (2-\gamma)} r^{2-\gamma}}$, which is a positive radial  solution  of 
        $$ { \cal P}_k^- ( D^2 v) + \mu^\prime  v r^{-\gamma} =  0\qquad \hbox{ in } B(0,1)\setminus \{0\}\, .$$ 
Arguing by contradiction, let us suppose that the quotient function $\frac{u}{v}$ has a positive supremum in $B(0,1)\setminus \{0\}$, that is 
$$\sup_{B(0,1)\setminus \{0\}} { u \over v} = :\, \eta'>0\, .$$
Let us select $0<\eta< \eta'$ such that 
$\eta > \eta^\prime { \mu \over \mu^\prime}$. 

 We observe that the function $u-\eta v$ has a positive supremum in $B(0,1)\setminus \{0\}$. Let $\bar r\in [0,1)$ be any limit point of a maximizing sequence.  If $\bar r >0$, then we have 
       $$-\mu u( \bar r) {\bar r}^{-\gamma} \leq { \cal P}^-_k ( D^2 u) \leq { \cal P}_k^- ( D^2 (\eta v)) \leq -\mu^\prime \eta v( \bar r)
       {\bar r}^{-\gamma} \, ,$$ 
 which gives 
 $$
\mu^\prime \eta v( \bar r)\leq \mu u( \bar r) \leq \mu \eta' v(\bar r)
$$
in contradiction with the choice of $\eta$.

On the other hand, if $\bar r=0$, a contradiction can be reached by observing that
$${ \cal P}^+_k ( D^2 (u-\eta v) )\geq 
        { \cal P}_k^- ( D^2 u)-{ \cal P}_k^-( D^2 (\eta v))   \geq -\mu u r^{-\gamma}+ \mu ^\prime \eta v r^{-\gamma} \geq ( \mu^\prime\eta  -\mu\eta^\prime ) v r^{-\gamma} >0\, .$$
 Indeed, in such a case, Remark \ref{remsign}  yields that $(u-\eta v)^\prime(r) \geq 0$ for $r$ small enough, and this  contradicts the fact that $u-\eta v$ has a maximizing sequence converging to zero. 
\medskip
       
\emph{Case 2.} For  $\mu^\prime > \max \{ \mu, \tau k\}$,    let us consider the function  $v(r)=  r^{-\frac{\mu^\prime }{k}}$, which satisfies 
         $$ { \cal P}_k^-( D^2 v) + \mu^\prime v r^{-2} = 0\qquad \hbox{ in } B(0,1)\setminus \{0\}\, .$$ 
 By the assumption  on $u$ and the choice of $\mu'$,  one has $\lim_{ r \rightarrow 0} { u(r) \over v (r)} = 0$. Reasoning as above, we suppose by contradiction that 
 $$\eta^\prime\, : = \sup_{B(0,1)\setminus \{0\}}  { u \over v}>0\, . $$ 
 Then  for any $0<\eta < \eta^\prime$, with  $\eta \mu^\prime > \mu \eta^\prime$,  the supremum of $u-\eta v $ is positive  and achieved at some point $\bar r>0$. This yields a contradiction as in the previous case.
 \medskip
         
\emph{Case 3.} For $\mu^\prime >  \max \{ \mu, ck ( \gamma-2)\}$, we consider the function $v(r)=e^{-\frac{\mu'}{k (2-\gamma)} r^{2-\gamma}}$ and we argue  as in the previous cases. 

  \end{proof}


  \section{The superlinear  problem in the critical case $\gamma = 2$.}
 \subsection{The superlinear  problem for ${ \cal P}^+_k$, with $k\geq 3$}  
In this section we are concerned with the existence and the asymptotic behavior near zero of positive  radial solutions $u\in {\cal C}^2(B(0,1)\setminus \{0\})$ of the  equation
\begin{equation}\label{superPk+} 
          { \cal P}_k^+ ( D^2 u) + \mu { u\over r^2} = u^p\, ,
 \end{equation}         
with $\mu>0$ and $p>1$. We will always assume that $k\geq 3$ and that $\mu < \bar \lambda_2= \left(\frac{k-2}{2}\right)^2$.

For this problem we can use the known results obtained for the semilinear equation having Laplace operator as principal part in dimension $k$, see \cite{Ci} and \cite{BDL2}. In particular, looking at the results in \cite{BDL2} for $\Lambda=\lambda=1$ and in dimension $k\geq 3$, under the assumption $\mu < \left(\frac{k-2}{2}\right)^2$ we showed that radial solutions of the equation
\begin{equation}\label{linear}
\Delta u +\mu \frac{u}{r^2}=u^p\qquad \hbox{ in } B(0,1)\setminus \{0\} \subset \R^k
\end{equation}
exist and can  have only specific  asymptotic behaviors at zero depending on the mutual values of $p$, $\mu$ and $k$. More precisely, using the notation
\begin{equation}\label{taupm}
\tau^{\pm} \, := \frac{k-2}{2} \pm \sqrt{\left(\frac{k-2}{2}\right)^2-\mu}\, ,
\end{equation}
we proved that there exist solutions $u$ of the semilinear equation \eqref{linear} satisfying
\begin{itemize}
\item[-] either $\lim_{r\to 0} u(r)r^{\frac{2}{p-1}}=0$, and in this case we proved that $u'\leq 0\leq u''$, see Proposition 4.2 in \cite{BDL2},

\item[-] or $\lim_{r\to 0} u(r)r^{\frac{2}{p-1}}>0$, which can occur if and only if either $p<1+\frac{2}{\tau^+}$ or $p>1+\frac{2}{\tau^-}$, and in this case an explicit solution is given by
$$
u(r)=\frac{K}{r^{\frac{2}{p-1}}}\, ,\qquad K=\left[ \left( \frac{2}{p-1}-\tau^-\right) \left( \frac{2}{p-1}-\tau^+\right)\right]^{\frac{1}{p-1}}\, .
$$
\end{itemize}
In any case, the constructed solutions are convex and non increasing functions of $r$. Hence, they satisfy in particular $u''\geq \frac{u'}{r}$, so that they are solutions of \eqref{superPk+}. By applying Theorem 1.1 in \cite{BDL2} we then deduce the following existence result.

\begin{theo}\label{exisuperPk+}
Let  us assume $0<\mu <\left( \frac{k-2}{2}\right)^2$ and $p>1$, and let  $p^\star= 1+\frac{2}{\tau^+}$ and  $p^{ \star \star}=1+\frac{2}{\tau^-}$,  with $\tau^\pm$ defined as in \eqref{taupm}. Then:    
\begin{enumerate}
     \item 
           if $p < p^{ \star }$,  then equation \eqref{superPk+}  has at least a  solution $u$ satisfying 
  \begin{equation}\label{2surp-1} 
   r^{2\over p-1}   u(r) \to  K \qquad \hbox{ as } r\to0\, ,
        \end{equation}
 at least a  solution $v$ satisfying
 \begin{equation}\label{tau+}
    r^{\tau^+}  v(r) \to  c_1 \qquad \hbox{ as } r\to 0\, ,
\end{equation}
and at least a solution $w$ satisfying
  \begin{equation}\label{tau-}
 r^{\tau^-}  w(r) \to  c_2  \qquad \hbox{ as } r\to0\, ,
        \end{equation}        
where   $K \, := \left[\left( {2\over p-1}-\tau^+\right) \left( {2\over p-1}-\tau^-\right)\right]^{1\over p-1}$ and  $c_1,\ c_2>0$ are suitable  constants;
      
   \item if  $p^\star\leq p< p^{ \star \star}$, then equation \eqref{superPk+} has at least a   solution $u$ satisfying
 \eqref{tau-};
 
 \item if $p = p^{ \star \star}$, then  equation \eqref{superPk+}  has at least a   solution $u$ satisfying 
\begin{equation}\label{tau-log}
 r^{\tau^-} ( -\log r)^{\tau^-/2}u(r) \to \bar K \qquad \hbox{ as } r\to 0\, ,
 \end{equation}
        where 
$\bar K = \left[ \tau^- \sqrt{\left(\frac{k-2}{2}\right)^2-\mu}\right]^{\tau^-/2}$;
        
        \item if  $p> p^{\star\star} $,  then  equation \eqref{superPk+}  has at least a    solution $u$ satisfying (\ref{2surp-1}).
 \end{enumerate}
\end{theo}

Conversely, let us show that the asymptotic behaviors identified  in Theorem \ref{exisuperPk+} are the only possible ones for \emph{any} solution $u$ of equation \eqref{superPk+}.  We are going to apply some of the arguments used  in the proof of Theorem 1.2 in \cite{BDL2}, which will be recalled for the reader's convenience in the  following preliminary result.

\begin{lemme}\label{Pk+linear}
Let $u\in {\cal C}^2(B(0,1)\setminus \{0\})$ be a positive radial solution of \eqref{superPk+}. Then
\begin{itemize}
\item[{\rm (i)}] $\limsup_{r \to 0} u(r)r^{\frac{2}{p-1}}<+\infty$;
\item[{\rm (ii)}] if  $\limsup_{r \to 0} u(r)r^{\frac{2}{p-1}}>0$, then  either $p<p^*$ or $p>p^{**}$ and $u$ satisfies \eqref{2surp-1}; 
\item[{\rm (iii)}] if  $\lim_{r \to 0} u(r)r^{\frac{2}{p-1}}=0$, then $u$ is a radial solution of the semilinear equation \eqref{linear} in $B(0,r_0)\setminus \{0\}\subset \R^k$ for some $r_0>0$ sufficiently small.
\end{itemize}
\end{lemme}
 \begin{proof} (i) A direct computation shows that, for $C>0$ sufficiently large (actually, for any $C>0$ if $p^*\leq p\leq p^{**}$), the function 
 $C r^{-2\over p-1}$ is a super-solution of  equation \eqref{superPk+} .  Hence, 
if $u(r)r^{\frac{2}{p-1}}$ is bounded along a decreasing sequence $\{r_n\}$ converging to zero, then, by applying  the standard comparison principle in every  annulus  $B_{r_{n}}\setminus\overline{B}_{r_{n+1}}$, we deduce 
$$u(r)\leq C r^{-2\over p-1} \quad \hbox{ for } 0<r\leq r_1\, .
$$ 
Thus, arguing by contradiction, if  statement (i) is false,  then we have
\begin{equation}\label{unbounded}
\lim_{r\to 0} u(r) r^{\frac{2}{p-1}}= +\infty\, .
\end{equation}   
This implies that, for $r$ small enough, one has
$$
\mathcal{P}^+_k(D^2u)= u^p -\mu \frac{u}{r^2} \geq { u^p \over 2}>0\, .$$ 
By Remark \ref{remsign}, it then follows that $u'$ has constant sign in a right neighborhood of zero and, therefore, by \eqref{unbounded}, $u'(r)\leq 0$ for $r$ small enough. We then have 
$$ u^{ \prime \prime} \geq { u^p\over 2}\, ,$$
 and also
          $$u^{ \prime \prime} u^\prime \leq {u^p u^\prime \over 2}\, .$$
           Integrating, one gets 
           $$ (u^\prime )^2(r) - c  u^{p+1} (r) \geq  (u^\prime )^2(r_0) - c  u^{p+1} (r_0)$$
            for $r < r_0$ and for some $c>0$. Using $u^{p+1}(r)\rightarrow +\infty$ as $r\to 0$, for $r$ small enough we get
            $$(u^\prime)^2 \geq \frac{c}{2} u^{p+1}$$
            or, equivalently,
            $$ {-u^\prime \over u^{ p+1\over 2}} \geq c_1>0\, .$$
This yields
             $$ u(r)^{ 1-p \over 2} \geq c_2 r\, ,$$
 for some $c_2>0$: a contradiction to \eqref{unbounded}. 
\medskip

(ii) By applying the comparison principle Theorem \ref{comparison} and by using exactly the same supersolutions constructed in Proposition 4.3 and Proposition 4.4 of \cite{BDL2}, we can prove that if $p^*\leq p\leq p^{**}$ then $\lim_{r\to 0} u(r)r^{\frac{2}{p-1}}=0$. Hence, under the assumption of statement (ii), we deduce that either $p<p^*$ or $p>p^{**}$. Moreover, by using again Theorem \ref{comparison} and the barrier functions exhibited in the proof of statement 2. of Proposition 4.2 of \cite{BDL2}, we obtain the asymptotic formula \eqref{2surp-1}.
\medskip

(iii) We need to prove that $u''\geq \frac{u'}{r}$ for $r$ sufficiently small, under the assumption $\lim_{r \to 0} u(r)r^{\frac{2}{p-1}}=0$.
We observe that, in this case, by equation \eqref{superPk+},   for $r$ sufficiently small one has
 $$
 {\cal P}^+_k (D^2u)= \frac{u}{r^2} ( r^2 u^{p-1}-\mu ) <0\, .
 $$
Hence, by Lemma \ref{lem1}, we have that $u'\leq 0$ for $r$ small enough. This in turn implies that the negative function $g(r)=u(r)^p-\mu \frac{u}{r^2}$ satisfies
$$
g'(r)=  u'\left( p u^{p-1}-\frac{\mu}{r^2}\right) +2 \mu \frac{u}{r^3} >0\, ,
$$
and the conclusion follows by Lemma \ref{lem1.1}. 
\end{proof}

By using the above lemma and by applying Theorem 1.2 in \cite{BDL2}, we immediately deduce the following classification result.

\begin{theo}\label{asymPk+}
Let  $u\in {\cal C}^2(B(0,1)\setminus \{0\})$ be a positive radial solution of \eqref{superPk+} with $p>1$ and $0<\mu <\left(\frac{k-2}{2}\right)^2$. Then, using the same notations as in Theorem \ref{exisuperPk+}, one has
\begin{enumerate}

\item if $1<p < p^{ \star }$, then  $u$ satisfies either \eqref{2surp-1} or  \eqref{tau+} or 
 \eqref{tau-};

\item if  $p^\star \leq p < p^{ \star \star}$,  then    $u$ satisfies    \eqref{tau-};  

\item if $p = p^{ \star\star}$, then $u$ satisfies \eqref{tau-log};

\item  if $p>p^{\star \star}$, then $u$ satisfies \eqref{2surp-1}.
 \end{enumerate}            
             \end{theo}


 \subsection{The superlinear problem for ${\cal P}_k^-$, with $2\leq k\leq N-1$}
In this section we focus on the equation
\begin{equation}\label{Pkmsuper}
 { \cal P}^-_k ( D^2 u) + \mu \frac{u}{r^{2}} = u^p \quad \hbox{ in } B(0,1)\setminus \{0\}\, ,
  \end{equation}
with $N-1\geq k\geq 2$, $\mu>0$ and $p>1$. Our goal is to prove the existence of radial positive solutions $u\in C^2(B(0,1)\setminus\{0\})$ and to classify their asymptotic behavior around zero in dependence of the mutual values of the parameters $k\, ,\ \mu$ and $p$.

Let us start with a preliminary result which collects some properties of solutions.

\begin{lemme}\label{prelprop}
Let $u\in C^2(B(0,1)\setminus\{0\})$  be a solution of \eqref{Pkmsuper} with $N-1\geq k\geq 2$, $\mu>0$ and $p>1$. Then:
\begin{itemize}
\item[{\rm (i)}] $\limsup_{r\to 0} u(r)=+\infty$;

\item[{\rm (ii)}] $\limsup_{r\to 0} u(r)r^{\frac{2}{p-1}}<+\infty$;

\item[{\rm (iii)}] if $\frac{\mu}{k}<\frac{2}{p-1}$, then the function $u(r)r^{\frac{\mu}{k}}$ is monotone non decreasing for $0<r<1$;

\item[{\rm (iv)}] if $\frac{\mu}{k}=\frac{2}{p-1}$, then the function $u(r)r^{\frac{2}{p-1}}(-\ln r)^{\frac{1}{p-1}}$ is bounded for $0<r<1$;

\item[{\rm (v)}] if $\limsup_{r\to 0} u(r)r^{\frac{2}{p-1}}>0$, then $\frac{\mu}{k} > \frac{2}{p-1}$ and $\liminf_{r\to 0} u(r)r^{\frac{2}{p-1}}>0$;

\item[{\rm (vi)}] if $\lim_{r\to 0} u(r)r^{\frac{2}{p-1}}=0$, then $u$ satisfies $u'\leq 0$ and $u''-\frac{u'}{r}\geq 0$ for $r>0$ sufficiently small.
\end{itemize}
\end{lemme}
\begin{proof}
(i)  By contradiction, let us assume that $u$ is bounded. Then,   for $r$ small enough, one has
 $$ { \cal P}_k^-( D^2 u)= u\left ( u^{p-1}-\frac{\mu}{r^2}\right)< 0\, ,$$
which yields, by Lemma \ref{lem1},  $u^\prime \leq 0$ in a neighborhood of zero. This in turn implies that the function $g(r)= u^p(r)-\mu \frac{u(r)}{r^2}$ satisfies
 $$g^\prime(r)=u^\prime \left( p u^{ p-1}-\frac{\mu}{r^2}\right) + 2 \mu \frac{u}{r^3}>0\, ,$$
so that, by Lemma \ref{lem1.1}, $u$ is such that $u''-\frac{u'}{r}$ has constant sign in a neighborhood of zero. If $u''-\frac{u'}{r}\leq0$, we get that, in a neighborhood of zero, $u$ is a bounded solution of
$$
u''+(k-1) \frac{u'}{r}= u^p-\mu \frac{u}{r^2}\, ,$$
in contradiction with Proposition 4.1 of \cite{BDL2}. On the other hand, if $u''-\frac{u'}{r}\geq0$, then, for $r>0$ small enough, $u$ satisfies
$$
k \frac{u'}{r}=u\left ( u^{p-1}-\frac{\mu}{r^2}\right)\leq - \frac{c}{r^2}\, ,
$$
for some $c>0$, again contradicting  the boundedness of $u$. 
\medskip

(ii) We observe that, whatever $\mu$ is,  the function 
 $C r^{-\frac{2}{p-1}}$ is a super-solution of  equation \eqref{Pkmsuper}  for $C>0$ sufficiently large. Hence, arguing by contradiction as in the proof of Lemma \ref{Pk+linear} (i), we get that, if statement (ii) fails, then
$$
\lim_{r\to 0} u(r)r^{\frac{2}{p-1}}=+\infty\, .
$$
This implies that, for $r$ small enough, $u$ satisfies
$$
\mathcal{P}_k^-(D^2u)= \frac{u}{r^2} \left( u^{p-1}r^2 -\mu\right)>0\, .$$ 
Then, by Remark \ref{remsign},  we obtain $u^\prime \geq 0$ and then  $u$  bounded, in contradiction with statement (i). 
\medskip

(iii) Let $0<r_0<1$ be fixed. For any $\epsilon>0$, the function
$$
v(r)=\frac{\epsilon}{r^{\frac{2}{p-1}}}+ \frac{u(r_0)r_0^{\frac{\mu}{k}}}{r^{\frac{\mu}{k}}}
$$
is convex and decreasing, and, as a radial function, it satisfies
$$
{\cal P}_k^-(D^2v)+\mu \frac{v}{r^2}= k \frac{v'}{r}+\mu \frac{v}{r^2}=\epsilon \frac{\mu-\frac{2k}{p-1}}{r^{\frac{2p}{p-1}}}<0<v^p\, .
$$
Moreover, by statement (ii), $u$ satisfies $u\leq c r^{-\frac{2}{p-1}}$. Hence,  $u$ and $v$ satisfy the assumptions of   Theorem \ref{comparison} and, since $u(r_0)\leq v(r_0)$, we get $u(r)\leq v(r)$ for all $0<r\leq r_0$. Letting $\epsilon \to 0$ yields the conclusion.
 
\medskip

(iv) Let us consider in this case, for any $\epsilon\, ,\ C>0$, the function
$$
v(r)=\frac{\epsilon}{r^{\frac{2}{p-1}}}+ \frac{C}{r^{\frac{2}{p-1}}(-\ln r)^{\frac{1}{p-1}}}\, .
$$
For $r_0$ sufficiently small and for $0<r<r_0$, $v$ satisfies $v'\leq 0\leq v''$ and
$$
{\cal P}_k^-(D^2v)+\mu \frac{v}{r^2} = k \frac{v'}{r}+\mu \frac{v}{r^2}=\frac{\mu \, C}{2 r^{\frac{2p}{p-1}}(-\ln r)^{\frac{p}{p-1}}}\leq v^p
$$
if $C$ is large enough, independently of $\epsilon>0$. If, moreover, $C$  is chosen in such a way that $v(r_0)\geq u(r_0)$, we can apply Theorem \ref{comparison} as before in order to conclude $u(r)\leq v(r)$ for $0<r\leq r_0$. Letting $\epsilon \to 0$, we get the conclusion.

\medskip

(v) By statements (iii) and (iv), we deduce that if $\frac{\mu}{k}\leq \frac{2}{p-1}$, then $\lim_{r \to 0}u(r) r^{\frac{2}{p-1}}=0$ . Hence, if $\limsup_{r \to 0}u(r) r^{\frac{2}{p-1}}>0$, then, necessarily, $\frac{\mu}{k}> \frac{2}{p-1}$ and there exist a decreasing sequence $\{ r_n\}$ converging to zero and some positive constant $l$ such that
$$
u(r_n) r_n^{\frac{2}{p-1}}\geq l\quad \hbox{ for all } n\geq 1\, .
$$
Let us consider the function $w(r)=\frac{c}{r^{\frac{2}{p-1}}}$, with $c=\min \left\{ l, \left(\mu -\frac{2k}{p-1}\right)^{\frac{1}{p-1}}\right\}$. Then, $w$ is a convex and decreasing function of $r>0$ satisfying, as a radial function,
$$
{\cal P}_k^-(D^2w)+\mu \frac{w}{r^2} = k \frac{w'}{r}+\mu \frac{w}{r^2}= \frac{c}{r^{\frac{2p}{p-1}}}\left( \mu-\frac{2k}{p-1}\right)\geq w^p\, .
$$
Since $w(r_n)\leq u(r_n)$ for all $n\geq 1$, by applying the standard comparison principle in each annular domain $B(0,r_n)\setminus B(0, r_{n+1})$, we finally deduce $u(r)\geq w(r)$ for $0<r<r_1$, which implies $\liminf_{r \to 0}u(r) r^{\frac{2}{p-1}}\geq c>0$.
\medskip

(vi) Let us argue as in the proof of statement (i). Under the assumption $\lim_{r\to 0} u(r)r^{\frac{2}{p-1}}=0$, we have that $u$ satisfies $\mathcal{P}^-_k(D^2u)=g(r)$ with the function $g(r)=u^p-\mu \frac{u}{r^2}$ such that $g(r)<0<g'(r)$ for $r$ sufficiently small. Then, by Lemma \ref{lem1.1}, we deduce that $u'\leq 0$ and $u''-\frac{u'}{r}$ has constant sign in a neighborhood of zero. We observe that if $u''\leq \frac{u'}{r}$, then $u$ is concave, hence  bounded, in contradiction with statement (i). Thus, one has $u''\geq \frac{u'}{r}$ for $r$ small enough.

\end{proof}
 
We are now ready to prove the main result of this section.
\begin{theo}\label{Pk-gamma=2} Let $N-1\geq k\geq 2$, $\mu>0$ and $p>1$ be given.

\begin{enumerate}
\item If $p<1+\frac{2k}{\mu}$ {\rm (}i.e. $\frac{\mu}{k}<\frac{2}{p-1}${\rm )},  then a radial function $u\in {\cal C}^2(B(0,1)\setminus \{0\}$ is a solution of equation \eqref{Pkmsuper} if and only if for $r>0$ sufficiently small $u(r)$ is of the form
\begin{equation}\label{case1}
u(r) =\frac{1}{\left( c\, r^{\frac{\mu (p-1)}{k}}- \frac{r^2}{\frac{2k}{p-1}-\mu}\right)^{\frac{1}{p-1}}}\, , 
\end{equation}
for some $c>0$. In particular, any solution $u$ satisfies 
$$
\lim_{r\to 0} u(r)r^{\frac{\mu}{k}}=c^{\frac{1}{p-1}}>0\, .
$$   
         
 \item  If $p=1+\frac{2k}{\mu}$ {\rm (}i.e. $\frac{\mu}{k}=\frac{2}{p-1}${\rm )},  then a radial function $u\in {\cal C}^2(B(0,1)\setminus \{0\}$ is a solution of equation \eqref{Pkmsuper} if and only if for $r>0$ sufficiently small $u(r)$ is of the form
 \begin{equation}\label{case2}
u(r)= \frac{1}{r^{\frac{2}{p-1}}\left(c+\frac{(p-1)}{k}(-\ln r)\right)^{\frac{1}{p-1}}}\, ,
 \end{equation}
 for some $c\in \R$. In particular, any solution $u$ satisfies 
$$
\lim_{r\to 0} u(r)r^{\frac{2}{p-1}}(-\ln r)^{\frac{1}{p-1}}=\left(\frac{k}{p-1}\right)^{\frac{1}{p-1}} \, .
$$   
 \item  If $p>1+\frac{2k}{\mu}$ {\rm (}i.e. $\frac{\mu}{k}>\frac{2}{p-1}${\rm )},  then a radial function $u\in {\cal C}^2(B(0,1)\setminus \{0\}$ is a solution of equation \eqref{Pkmsuper} if and only if for $r>0$ sufficiently small $u(r)$ is of the form
\begin{equation}\label{case3}
u(r) =\frac{1}{\left( c\, r^{\frac{\mu (p-1)}{k}}+\frac{r^2}{\mu -\frac{2k}{p-1}}\right)^{\frac{1}{p-1}}}\, , 
\end{equation}
for some $c\in \R$. In particular, any solution $u$ satisfies 
$$
\lim_{r\to 0} u(r)r^{\frac{2}{p-1}}=\left(\mu -\frac{2k}{p-1}\right)^{\frac{1}{p-1}}\, .$$            
          
\end{enumerate}
  \end{theo}
\begin{proof} \emph{Case 1.}  A direct computation shows that each function $u$ defined in \eqref{case1} is a solution of 
\begin{equation}\label{firstorder}
   k \frac{u'}{r}+\mu \frac{u}{r^2}=u^p
\end{equation}
defined in some right neighborhood of zero. Moreover, it is not difficult to check that $u''\geq \frac{u'}{r}$, hence any such $u$ is a radial solution of equation   \eqref{Pkmsuper} defined in $B(0,1)\setminus \{0\}$ if $c>0$ is chosen large enough.

Conversely, let $u$ be a radial solution of \eqref{Pkmsuper}. Under the current assumption $\frac{\mu}{k}<\frac{2}{p-1}$, by Lemma \ref{prelprop} (iii), we have that   $\lim_{r\to 0}u(r)r^{\frac{2}{p-1}}=0$. Hence, by Lemma \ref{prelprop} (vi), we deduce that $u$ satisfies the first order equation \eqref{firstorder} in an interval $(0,r_0)$ with $r_0<1$ small enough. By uniqueness of solution for equation \eqref{firstorder} we then conclude that $u$ has the expression given by \eqref{case1} with
\begin{equation}\label{c}
c=\frac{r_0^{2-\frac{\mu (p-1)}{k}}}{\frac{2k}{p-1}-\mu}+\frac{1}{\left( u(r_0) r_0^{\frac{\mu}{k}}\right)^{p-1}}\, .
\end{equation}
\medskip

\emph{Case 2.}  We argue as in the previous case, first by observing that functions given by \eqref{case2} are solutions of the first order equation \eqref{firstorder} and they satisfy $u''\geq \frac{u'}{r}$. Hence, formula \eqref{case2} with $c\in \R$ provides solutions of equation \eqref{Pkmsuper}. Conversely, given  any solution $u$ of \eqref{Pkmsuper}, we apply Lemma \eqref{prelprop} (iv) and (vi) in order to deduce that there exists $r_0>0$ sufficiently small such that, for $0<r\leq r_0$, $u$ is of the form \eqref{case2} with
$$
c=\frac{p-1}{k} \ln r_0+ \frac{1}{u(r_0)^{p-1}r_0^2}\, .
$$

\medskip

\emph{Case 3.}  As in the previous cases, a direct computation shows that any function $u$ given by \eqref{case3} is a radial solution of \eqref{Pkmsuper}. 

Conversely, let $u$ be any solution of \eqref{Pkmsuper}. 
We claim that $u$ satisfies
$\limsup_{r\to 0} u(r)r^{\frac{2}{p-1}}>0$. Indeed, if not, then, again by Lemma \ref{prelprop} (vi), we deduce that $u$ is a solution of equation \eqref{firstorder} for $0<r<r_0$, with $r_0>0$ small enough. Then, $u$ has an expression given by \eqref{case3}, but none of these functions satisfies $\lim_{r\to 0} u(r)r^{\frac{2}{p-1}}=0$. 

Hence, Lemma \ref{prelprop} (v) applies, yielding that, for any fixed $0<r_0<1$ there exist  $c_1, \ c_2>0$ depending on $r_0$ such that $u$ satisfies
$$
c_1r^{-\frac{2}{p-1}}\leq u(r) \leq c_2r^{-\frac{2}{p-1}} \quad \hbox{ for } 0<r\leq r_0\, .
$$
Thus, the functions $u$ and 
$$
v(r)= \frac{1}{\left( c\, r^{\frac{\mu (p-1)}{k}}+\frac{r^2}{\mu -\frac{2k}{p-1}}\right)^{\frac{1}{p-1}}}\, ,
$$
both radial solutions of equation \eqref{Pkmsuper} in $B(0,r_0)\setminus \{ 0\}$, can be compared by means of Theorem \ref{comparison}. By choosing $c\in \R$ as in \eqref{c}, we conclude that $u\equiv v$  in $B(0,r_0)\setminus \{ 0\}$.

\end{proof}

    \end{document}